\documentclass{amsart}

\newtheorem{theorem}{Theorem}
\newtheorem{lemma}[theorem]{Lemma}
\newtheorem{proposition}[theorem]{Proposition}

\theoremstyle{definition}

\newtheorem{remark}[theorem]{Remark}

\usepackage{amscd,amssymb}

\begin{document}

\title[Algebras generated by two quadratic elements]
{Algebras generated by two quadratic elements}

\author[{V. Drensky, J. Szigeti, L. van Wyk}]
{Vesselin Drensky, Jen\H{o} Szigeti, Leon van Wyk}
\address{Institute of Mathematics and Informatics, Bulgarian Academy of Sciences,
1113 Sofia, Bulgaria}
\email{drensky@math.bas.bg}
\address{Institute of Mathematics, University of Miskolc, Miskolc, Hungary 3515}
\email{jeno.szigeti@uni-miskolc.hu}
\address{Department of Mathematical Sciences, Stellenbosch University\\
P/Bag X1, Matieland 7602, Stellenbosch, South Africa}
\email{lvw@sun.ac.za}
\thanks{The second author was supported by OTKA of Hungary No. K61007.}
\thanks{The third author was supported by the National Research Foundation of
South Africa under Grant No. UID 61857. Any opinion, findings and conclusions
or recommendations expressed in this material are those of the authors and therefore
the National Research Foundation does not accept any liability in regard thereto.}

\begin{abstract}
Let $K$ be a field of any characteristic and let $R$ be an algebra
generated by two elements satisfying quadratic equations. Then $R$ is a homomorphic image of
$F=K\langle x,y\mid x^2+ax+b=0,y^2+cy+d=0\rangle$ for suitable $a,b,c,d\in K$.
We establish that $F$ can be embedded into the $2\times 2$ matrix algebra
$M_2(\overline{K}[t])$ with entries from the polynomial algebra $\overline{K}[t]$
over the algebraic closure of $K$ and that $F$ and $M_2(\overline{K})$ satisfy the same
polynomial identities as $K$-algebras. When the quadratic equations have double zeros,
our result is a partial case of more general results
by Ufnarovskij, Borisenko and Belov from the 1980's. When each of the equations has different zeros,
we improve a result of Weiss, also from the 1980's.
\end{abstract}

\maketitle

\section*{Introduction}

In this paper $K$ is an arbitrary field of any characteristic. All algebras are unital
and over $K$ or over its algebraic closure $\overline{K}$.

Initially this project was motivated by results on the minimum number
of idempotents needed to generate various kind of algebras, see e.g.
Roch and Silbermann \cite{RS}, Krupnik \cite{K},
Kelarev, van der Merwe and van Wyk \cite{KMW},
van der Merwe and van Wyk \cite{MW}, Goldstein and Krupnik \cite{GK}.
In particular, Krupnik \cite{K} exhibited three idempotent matrices which
generate $M_n(K)$. He also
showed that any algebra generated
by two idempotents satisfies the standard identity of degree
four, and hence $M_{n}(K)$ cannot be generated by two idempotent matrices if $n\geq 3$.
Independently, Weiss \cite{W1, W2} established that if an infinite dimensional algebra
$R$ is generated by two idempotents and $K$ is of characteristic different from 2, then
$R$ is isomorphic to a subalgebra of $M_2(K[v])$.

It is worthy to study algebras generated by idempotents and nilpotent elements
and this was another motivation for our project.
Van der Merwe and van Wyk \cite[Theorem 2.1]{MW} showed that
the algebra $M_n(K)$, $n\geq 2$, over an arbitrary field $K$ can be generated by
an idempotent matrix and a nilpotent matrix of index $n$.

Our third motivation came from the well developed theory of (noncommutative) monomial algebras.
For an ideal $J$ of the free associative algebra $K\langle x_1,\ldots,x_d\rangle$
generated by a finite set of monomials Ufnarovskij \cite{U1} associated an oriented graph
and described the growth of the factor algebra $K\langle x_1,\ldots,x_d\rangle/J$
in terms of paths in the graph. In particular, the growth of such an algebra is either
polynomial or exponential and its Hilbert series is rational. Borisenko \cite{Bo}
showed that a finitely presented monomial algebra is of polynomial growth if and only if it
satisfies a polynomial identity. Then for a suitable $n$ it can be embedded into the matrix algebra
$M_n(K[t])$ over the polynomial algebra $K[t]$. This allows to treat successfully the case when
the algebra has the presentation
\[
F_{p,q}\cong K\langle x,y\mid x^p=y^q=0\rangle,\quad p,q\geq 2,
\]
i.e., $F_{p,q}$ is generated by two elements and the only relations give the nilpotency of the generators.
If $q\geq 3$, then it is easy to see (and it follows also from \cite{U1}) that the monomials
$xyx$ and $xy^2x$ generate a free subalgebra of $F_{p,q}$. If $p=q=2$, the approach of \cite{U1, Bo} gives that
the algebra $F_{2,2}$ can be embedded into $M_n(K[t])$ for a suitable $n\leq 6$.
Further, Belov \cite{Be} developed the theory of monomial algebras
associated with infinite periodic words. If $w=x_{i_1}\cdots x_{i_m}$ is a monomial, then the algebra $A_w$
has a basis consisting of all finite subwords of the infinite word
\[
w^{\infty}=\ldots x_{i_1}\cdots x_{i_m}x_{i_1}\cdots x_{i_m}\ldots
\]
and all monomials which are not subwords of $w^{\infty}$ are equal to 0 in $A_w$.
Again, the algebra $A_w$ can be embedded into $M_n(K[t])$ and, additionally,
both algebras have the same polynomial identities.
Tracing step-by-step the proof of Belov, one sees that the algebras $F_{2,2}$ and $A_{xy}$ are isomorphic
and can be embedded into $M_2(K[t])$. An account of the work of Ufnarovskij, Borisenko and Belov
can be found in the survey
articles by Ufnarovskij \cite{U2} and Belov, Borisenko and Latyshev \cite{BBL}.
For a background on PI-algebras see e.g. \cite{DF}.

The purpose of this paper is to describe algebras generated by two elements satisfying quadratic equations.
We prove that
the algebra with presentation
\[
F=K\langle x,y\mid x^2+ax+b=0,y^2+cy+d=0\rangle
\]
for suitable $a,b,c,d\in K$ can be embedded into the $2\times 2$ matrix algebra
$M_2(\overline{K}[t])$ with entries from the polynomial algebra $\overline{K}[t]$
over the algebraic closure of $K$ and that $F$ and $M_2(\overline{K})$ satisfy the same
polynomial identities as $K$-algebras. Our embedding is similar to the embedding obtained
by the methods of Belov \cite{Be}.
As in \cite{W1, W2} we show that $F$ is a free module of rank 4 over its centre and
all proper homomorphic images of $F$ are
finite dimensional.

Our algebra $F$ is a free product of two two-dimensional algebras,
\[
F=K[x\mid x^2+ax+b=0]\ast K[y\mid y^2+cy+d=0].
\]
There is an obvious analogy of $F$ in group theory. The free product of two cyclic groups
\[
G=\langle x\mid x^p=1\rangle\ast \langle y\mid y^q=1\rangle,\quad p,q\geq 2,
\]
contains a free subgroup if $q\geq 3$, and is metabelian (solvable of class 2) if $p=q=2$,
see e.g. \cite[Problem 19, p. 195]{MKS}.

\section*{Main results}

Let $S$ be any $K$-algebra. Then $S$ may be considered as a $K$-subalgebra of the $\overline{K}$-algebra
$\overline{S}=\overline{K}\otimes_KS$ via the embedding $s\to 1\otimes s$, $s\in S$. In particular, we assume that
the free associative algebra $K\langle x_1,x_2,\ldots\rangle$ is naturally embedded into
$\overline{K}\langle x_1,x_2,\ldots\rangle$.
For a $K$-algebra $R$ we denote by $T_K(R)$ the T-ideal of $K\langle x_1,x_2,\ldots\rangle$
consisting of all polynomial identities of $R$. The corresponding notation when $R$ is a $\overline{K}$-algebra
is $T_{\overline{K}}(R)$.
The following fact is well known, see e.g. \cite[p. 12, Remark 1.2.9 (ii)]{DF}.

\begin{lemma}\label{Lemma 1}
If the field $K$ is infinite, $C$ is a commutative (unital) $K$-algebra, and the $K$-algebra $S$ satisfies
a polynomial identity, then $S$ and $C\otimes_KS$ have the same polynomial identities. In particular, $S$
and $\overline{S}=\overline{K}\otimes_KS$ have the same polynomial identities
and $T_K(S)=T_{\overline{K}}(\overline{S})\cap K\langle x_1,x_2,\ldots\rangle$.
\end{lemma}

We fix $\delta,\varepsilon=0$ or $1$ and elements $\alpha_1,\alpha_2,\beta_1,\beta_2\in \overline{K}$,
$\alpha_1,\beta_1\not=0$. We define the following $2\times 2$ matrices with entries in
the polynomial algebra $\overline{K}[t]$:
\[
X=\left(\begin{matrix}
\alpha_2+\delta\alpha_1&\alpha_1t\\
0&\alpha_2\\
\end{matrix}\right),\quad
Y=\left(\begin{matrix}
\beta_2&0\\
\beta_1t&\beta_2+\varepsilon\beta_1\\
\end{matrix}\right),
\]
\[
U=\left(\begin{matrix}
\delta&t\\
0&0\\
\end{matrix}\right),\quad
V=\left(\begin{matrix}
0&0\\
t&\varepsilon\\
\end{matrix}\right),
\]
i.e., $X=\alpha_1U+\alpha_2I$, $Y=\beta_1V+\beta_2I$, where $I$ is the identity $2\times 2$ matrix.
Direct computations give
\[
U^2=\delta U,\quad V^2=\varepsilon V,
\]
\[
XY=\left(\begin{matrix}
(\alpha_2+\delta\alpha_1)\beta_2+\alpha_1\beta_1t^2&\alpha_1t(\beta_2+\varepsilon\beta_1)\\
\alpha_2\beta_1t&\alpha_2(\beta_2+\varepsilon\beta_1)\\
\end{matrix}\right),
\]
\[
YX=\left(\begin{matrix}
\beta_2(\alpha_2+\delta\alpha_1)&\beta_2\alpha_1t\\
\beta_1t(\alpha_2+\delta\alpha_1)&\beta_1\alpha_1t^2+(\beta_2+\varepsilon\beta_1)\alpha_2\\
\end{matrix}\right),
\]
\[
[X,Y]=XY-YX=\alpha_1\beta_1t\left(\begin{matrix}
t&\varepsilon\\
-\delta&-t\\
\end{matrix}\right),\quad
[X,Y]^2=\alpha_1^2\beta_1^2t^2(t^2-\delta\varepsilon)I.
\]

\begin{proposition}\label{Proposition 2}
The $K$-subalgebra $R$ generated by $X$ and $Y$ in $M_2(\overline{K}[t])$ satisfies the same polynomial
identities as the $K$-algebra $M_2(\overline{K})$.
\end{proposition}

\begin{proof}
Since the algebra $R$ is a $K$-subalgebra of $M_2(\overline{K}[t])$, it satisfies all polynomial identities of
$M_2(\overline{K}[t])$. By Lemma \ref{Lemma 1} the $\overline{K}$-algebras $M_2(\overline{K}[t])$ and
$M_2(\overline{K})$ have the same polynomial identities. Hence $T_K(M_2(\overline{K}))\subseteq T_K(R)$.

First, let the base field $K$ be infinite. By Lemma \ref{Lemma 1} the $K$-algebras $R$ and $\overline{R}$
have the same polynomial identities. Hence, the inclusion in the opposite direction
$T_K(R)\subseteq T_K(M_2(\overline{K}))$ would follow if we
show that the $\overline{K}$-algebra
$M_2(\overline{K})$ is a homomorphic image of the $\overline{K}$-algebra $\overline{R}$.
It is sufficient to see that the matrices $I,X_0=X(t_0),Y_0=Y(t_0),X_0Y_0$
are linearly independent for some $t_0\in\overline{K}$. Equivalently, we may consider the matrices
$I, U_0=U(t_0), V_0=V(t_0),U_0V_0$.
Let $t_0\in\overline{K}$, $t_0\not=0,\pm 1$, and let for some $\xi_0,\xi_1,\xi_2,\xi_3\in\overline{K}$
\[
0=\xi_0I+\xi_1U_0+\xi_2V_0+\xi_3U_0V_0
=\left(\begin{matrix}
\xi_0+\delta\xi_1+t_0^2\xi_3&t_0(\xi_1+\varepsilon\xi_3)\\
t_0\xi_2&\xi_0+\varepsilon\xi_2\\
\end{matrix}\right).
\]
We derive consecutively
\[
\xi_2=0,\quad \xi_0=0,\quad \xi_1+\varepsilon\xi_3=0,\quad \delta\xi_1+t_0^2\xi_3=0.
\]
Since $\delta,\varepsilon=0,1$ and $t_0\not=0,\pm 1$ we obtain $t_0^2-\delta\varepsilon\not=0$.
Hence $\xi_1=\xi_3=0$ and
$I,X_0,Y_0,X_0Y_0$ are linearly independent and span $M_2(\overline{K})$ for any $t_0\not=0,\pm 1$.
Therefore $M_2(\overline{K})$ is a homomorphic image of $\overline{R}$.

Now, let the field $K$ be finite. We mimic the standard Vandermonde arguments used in theory of PI-algebras.
Let $f(x_1,\ldots,x_n)\in K\langle x_1,x_2,\ldots\rangle$ be a polynomial identity for the algebra $R$.
Hence $f(r_1,\ldots,r_n)=0$ for any $r_1,\ldots,r_n\in R$. We are interested in the case when
$r_1\in R$ has the form
\[
r_1=h_0I+h_1X+h_2Y+h_3XY,
\]
where $h_i=h_i([X,Y]^2)$, $i=0,1,2,3$, are polynomials in $[X,Y]^2=\alpha_1^2\beta_1^2t^2(t^2-\delta\varepsilon)I$.
In this way, $f(r_1,\ldots,r_n)=0$
is an evaluation of $g=f(x_{10}+x_{11}+x_{12}+x_{13},x_2,\ldots,x_n)$.
We write $g$ as
\[
g=g(x_{10},x_{11},x_{12},x_{13},x_2,\ldots,x_n)=\sum_{j=0}^mg_j(x_{10},x_{11},x_{12},x_{13},x_2,\ldots,x_n),
\]
where $g_i(x_{10},x_{11},x_{12},x_{13},x_2,\ldots,x_n)$ is the homogeneous component of $g$
of degree $j$ in $x_{10}$.
Evaluating $f$ on
\[
r_1=[X,Y]^{2i}I+h_1X+h_2Y+h_3XY,\quad r_2,\ldots,r_n\in R,
\]
we obtain
\[
f(r_1,\ldots,r_n)=g([X,Y]^{2i}I,h_1X,h_2Y,h_3XY,r_2,\ldots,r_n)
\]
\[
=\sum_{j=0}^m[X,Y]^{2ij}g_j(I,h_1X,h_2Y,h_3XY,r_2,\ldots,r_n)
\]
\[
=\sum_{j=0}^m(\alpha_1^2\beta_1^2t^2(t^2-\delta\varepsilon))^{ij}g_j(I,h_1X,h_2Y,h_3XY,r_2,\ldots,r_n)=0.
\]
For $i=0,1,\ldots,m$ we obtain $m+1$ equations which form a linear homogeneous system with unknowns
$g_j(I,h_1X,h_2Y,h_3XY,r_2,\ldots,r_n)$, $j=0,1,\ldots,m$. Its determinant $\Delta$
is equal to the Vandermonde determinant
with entries
\[
v_{ij}=(\alpha_1^2\beta_1^2t^2(t^2-\delta\varepsilon))^{ij},\quad i,j=0,1,\ldots,m,
\]
and is different from zero because the elements $v_{1j}$ are pairwise different.
Therefore the zero solution
\[
g_j(I,h_1X,h_2Y,h_3XY,r_2,\ldots,r_n)=0,\quad j=0,1,\ldots,m,
\]
is the only solution of the system over the field of rational functions $K(t)$ and hence also over $K[t]$.
This means that the homogeneous component $g_j$ of degree $j$ in $x_{10}$
of the polynomial identity $g(x_{10},x_{11},x_{12},x_{13},x_2,\ldots,x_n)$ of $R$ vanishes when evaluated for
\[
x_{10}=I,x_{11}=h_1X,x_{12}=h_2Y,x_{13}=h_3XY,x_2=r_2,\ldots,x_n=r_n.
\]
With the same arguments we conclude that the multihomogeneous components
$g_{(i_0,i_1,i_2,i_3)}$ of $g$ of degree $(i_0,i_1,i_2,i_3)$ in $(x_{10},x_{11},x_{12},x_{13})$
vanish when evaluated for
\[
x_{10}=I,x_{11}=X,x_{12}=Y,x_{13}=XY,x_2=r_2,\ldots,x_n=r_n.
\]
Since $I,X,Y,XY$ are linearly independent in $M_2(K(t))$, they form a basis over
$K(t)$ and every element $s_1$ in $M_2(K(t))$ has the form
\[
s_1=h_0I+h_1X+h_2Y+h_3XY,\quad h_0,h_1,h_2,h_3\in K(t).
\]
Hence
\[
f(s_1,r_2,\ldots,r_n)= f(h_0I+h_1X+h_2Y+h_3XY,r_2,\ldots,r_n)
\]
\[
=g(h_0I,h_1X,h_2Y,h_3XY,r_2,\ldots,r_n)
\]
\[
=\sum g_{(i_0,i_1,i_2,i_3)}(h_0I,h_1X,h_2Y,h_3XY,r_2,\ldots,r_n)
\]
\[
=\sum h_0^{i_0}h_1^{i_1}h_2^{i_2}h_3^{i_3}g_{(i_0,i_1,i_2,i_3)}(I,X,Y,XY,r_2,\ldots,r_n)=0
\]
for any $s_1\in M_2(K(t))$, $r_2,\ldots,r_n\in R$. Continuing in this way, we conclude that
$f(s_1,s_2,\ldots,s_n)=0$ for any $s_1,\ldots,s_n\in M_2(K(t))$ and $f(x_1,\ldots,x_n)$ is a polynomial identity for
$M_2(K(t))$. Hence $T_K(R)=T_K(M_2(K(t)))$.
Since the field $K(t)$ is infinite, Lemma \ref{Lemma 1} gives
\[
T_{K(t)}(M_2(K(t)))=T_{K(t)}(M_2(\overline{K}(t))),\quad
T_{\overline{K}}(M_2(\overline{K}(t)))=T_{\overline{K}}(M_2(\overline{K})),
\]
and this implies $T_K(R)=T_K(M_2(\overline{K}))$ because
\[
T_K(R)=T_K(M_2(K(t)))=T_{K(t)}(M_2(K(t)))\cap K\langle x_1,x_2,\ldots\rangle
\]
\[
=T_{K(t)}(M_2(\overline{K}(t)))\cap K\langle x_1,x_2,\ldots\rangle
=T_{\overline{K}}(M_2(\overline{K}(t)))\cap K\langle x_1,x_2,\ldots\rangle
\]
\[
=T_{\overline{K}}(M_2(\overline{K}))\cap K\langle x_1,x_2,\ldots\rangle
=T_K(M_2(\overline{K})).
\]
\end{proof}

Now we shall prove the main result of our paper.

\begin{theorem}\label{Main Theorem}
Let $K$ be a field of any characteristic and let $a,b,c,d$ be arbitrary elements of $K$.
Then the algebra
\[
F=K\langle x,y\mid x^2+ax+b=0,y^2+cy+d=0\rangle
\]
can be embedded into $M_2(\overline{K}[t])$. The algebras $F$ and $M_2(\overline{K})$ satisfy the same
polynomial identities as $K$-algebras.
\end{theorem}

\begin{proof}
{\bf Step 1.}
The equation $f(x)=x^2+ax+b=0$ has two zeros $\eta_1,\eta_2$ in an extension of $K$ and $f(x)=(x-\eta_1)(x-\eta_2)$
in $\overline{K}[x]$. If $\eta_1=\eta_2$ we change the variable
$x$ in $\overline{K}[x]$ by $x=u+\eta_1$ and obtain that
$f(x)=u^2$. Similarly, if $\eta_1\not=\eta_2$, we change $x$ by
$x=\alpha_1u+\alpha_2$, $\alpha_1=\eta_2-\eta_1,\alpha_2=\eta_1$ in $\overline{K}$.
Then $f(x)=\alpha_1^2(u^2-u)=(a^2-4b)(u^2-u)$.
Hence, working in the $\overline{K}$-algebra
\[
\overline{F}=\overline{K}\langle x,y\mid x^2+ax+b=0,y^2+cy+d=0\rangle
\]
which contains $F$, we may assume that it is generated by $u$ and $v$
such that either $u^2=0$ or $u^2=u$ and similarly either $v^2=0$ or $v^2=v$.
Hence  $\overline{F}=\overline{K}\otimes_KF$ has the presentation
\[
\overline{F}=\overline{K}\langle u,v\mid u^2=\delta u,v^2=\varepsilon v\rangle,
\]
where $\delta,\varepsilon=0,1$.
As a $\overline{K}$-vector space $\overline{F}$ is spanned on the monomials
\[
1,\quad (uv)^p,(vu)^p,p\geq 1,\quad (uv)^qu,(vu)^qv,\quad q\geq 0.
\]
(It follows from the general theory of free products or of Gr\"obner bases that
these elements form a $\overline{K}$-basis of $\overline{F}$.)

{\bf Step 2.}
We consider the matrices
\[
U=\left(\begin{matrix}
\delta&t\\
0&0\\
\end{matrix}\right),\quad
V=\left(\begin{matrix}
0&0\\
t&\varepsilon\\
\end{matrix}\right)
\]
in $M_2(\overline{K}[t])$. As in Step 1 we choose $\alpha_1,\alpha_2,\beta_1,\beta_2\in\overline{K}$,
$\alpha_1,\beta_1\not=0$,
such that the matrices
\[
X=\alpha_1U+\alpha_2I,\quad Y=\beta_1V+\beta_2I
\]
satisfy the equations $X^2+aX+bI=0$, $Y^2+cY+dI=0$.
Let $R$ be the $K$-subalgebra of $M_2(\overline{K}[t])$ generated by $X$ and $Y$.
The $\overline{K}$-subalgebra $\overline{R}$ of $M_2(\overline{K}[t])$ is generated also by $U$ and $V$.
Since $U^2=\delta U$, $V^2=\varepsilon V$ the mapping
\[
u\to U,\quad v\to V
\]
extends to a homomorphism $\varphi$ of
the $\overline{K}$-algebra $\overline{F}$ to the $\overline{K}$-algebra $\overline{R}$.

{\bf Step 3.}
Direct computations give that
\[
(UV)^p=\left(\begin{matrix}
t^{2p}&\varepsilon t^{2p-1}\\
0&0\\
\end{matrix}\right)=t^{2(p-1)}UV,
\quad
(VU)^p=\left(\begin{matrix}
0&0\\
\delta t^{2p-1}&t^{2p}\\
\end{matrix}\right)=t^{2(p-1)}VU,
\]
\[
(UV)^qU=\left(\begin{matrix}
\delta t^{2q}&t^{2q+1}\\
0&0\\
\end{matrix}\right)=t^{2q}U,
\quad
(VU)^qV=\left(\begin{matrix}
0&0\\
t^{2q+1}&\varepsilon t^{2q}\\
\end{matrix}\right)=t^{2q}V.
\]
Hence the maximum degree in $t$ of the entries of a nonzero monomial of degree $k$ in $U,V$ is equal to $k$.

{\bf Step 4.}
If the homomorphism $\varphi:\overline{F}\to\overline{R}$ is not an isomorphism, the kernel
of $\varphi$ contains a nonzero element of $\overline{F}$
\[
f(u,v)=\alpha_0+\sum_{p\geq 1}(\alpha_{1p}(uv)^p+\alpha_{2p}(vu)^p)
+\sum_{q\geq 0}(\beta_{1q}(uv)^qu+\beta_{2q}(vu)^qv),
\]
$\alpha_0,\alpha_{ip},\beta_{iq}\in \overline{K}$. If $\deg_{u,v}(f)=2k$, then
at least one of the coefficients $\alpha_{1k}$ and $\alpha_{2k}$ is different from 0.
The diagonal entries of the matrix $f(U,V)$ contain, respectively, the monomials
$\alpha_{1k}t^{2k}$ and $\alpha_{2k}t^{2k}$. At least one of these monomials is nonzero
and does not cancel with any other monomial of the corresponding entries of $f(U,V)$.
Similarly, if $\deg_{u,v}(f)=2k+1$, then the entries of the other diagonal of the matrix $f(U,V)$ contain,
respectively, the monomials $\beta_{1k}t^{2k+1}$ and $\beta_{2k}t^{2k+1}$. Again, at least
one of these monomials is nonzero and does not cancel with other monomials at the same position.
This implies that $\varphi$ is an isomorphism and $\overline{F}$ can be embedded into
$M_2(\overline{K}[t])$.
This completes the proof in virtue of Proposition \ref{Proposition 2}.
\end{proof}

As in Weiss \cite{W1, W2} the algebra $F$ is a free module of rank 4 over its centre.
We can extend Theorem \ref{Main Theorem} to arbitrary infinite dimensional algebras
generated by two quadratic elements.

\begin{theorem}\label{module over centre}
Let
\[
F=K\langle x,y\mid x^2+ax+b=0,y^2+cy+d=0\rangle
\]

{\rm (i)} The centre $C(F)$ of $F$ is isomorphic to the polynomial algebra in one variable
and $F$ is a free $C(F)$-module of rank $4$ generated by $1,x,y,xy$.

{\rm (ii)} All proper ideals of $F$ are of finite codimension in $F$.
\end{theorem}

\begin{proof}
(i) We work with the matrix representation of $F$ as the $K$-subalgebra $R$ of $M_2(\overline{K}[t])$
generated by the matrices
\[
X=\left(\begin{matrix}
\alpha_2+\delta\alpha_1&\alpha_1t\\
0&\alpha_2\\
\end{matrix}\right),\quad
Y=\left(\begin{matrix}
\beta_2&0\\
\beta_1t&\beta_2+\varepsilon\beta_1\\
\end{matrix}\right).
\]
The Cayley-Hamilton theorem gives
\[
(X+Y)^2-(\text{tr}(X+Y))(X+Y)+\text{det}(X+Y)I=0,
\]
\[
(X+Y)^2-(2(\alpha_2+\beta_2)+\delta\alpha_1+\varepsilon\beta_1)(X+Y)
\]
\[
+((\alpha_2+\beta_2+\delta\alpha_1)(\alpha_2+\beta_2+\varepsilon\beta_1)-\alpha_1\beta_1t^2)I=0.
\]
Using that in $F$
\[
x^2+ax+b=0,\quad y^2+cy+d=0,
\]
which become in $R$
\[
X^2-(2\alpha_2+\delta\alpha_1)X+\alpha_2(\alpha_2+\delta\alpha_1)I=0,
\]
\[
Y^2-(2\beta_2+\varepsilon\beta_1)Y+\beta_2(\beta_2+\varepsilon\beta_1)I=0,
\]
we obtain
\[
XY+YX+cX+aY+f(t^2)I=0,\quad f(t^2)\in \overline{K}[t^2].
\]
Hence the element
\[
Z=XY+YX+cX+aY\in R
\]
belongs to the centre of $R$. The equations
\[
X^2=-(aX+bI),\quad Y^2=-(cY+dI),\quad YX=-(cY+dI)cX-aY-XY
\]
allow to express every element of $R$ as a linear combination of $I,X,Y,XY$
with coefficients from $K[Z]$, i.e., $R$ is a 4-generated $K[Z]$-module. It is easy to see that
this module is free and the centre of $R$ coincides with $K[Z]$.
Going back to $F$ we obtain that
\[
C(F)=K[z],\quad z=xy+yx+cx+ay
\]
and $F$ is a free $C(F)$-module freely generated by $1,x,y,xy$.
In the special case $X=U$, $Y=V$ the above formulas are quite simple:
\[
(U+V)^2-(\delta+\varepsilon)(U+V)+(\delta\varepsilon-t^2)I=0,
\]
\[
Z=XY+YX-\varepsilon U-\delta V=(t^2-\delta\varepsilon)I\in C(\overline{R})
\]
and the elements of $\overline{R}$ and the multiplication rules between them are given by
\[
W=\rho_0Z+\rho_1U+\rho_2V+\rho_3UV,\quad Z=(t^2-\delta\varepsilon)I, \rho_i\in \overline{K}[Z],
\]
\[
U^2=\delta U,\quad V^2=\varepsilon V,\quad
VU=-(U-\delta I)(V-\varepsilon I)+t^2I.
\]

(ii) If $J$ is an ideal of $R$, then $\overline{J}$ is an ideal of $\overline{R}$
and the codimension of $J$ in $R$ as a $K$-vector space and of $\overline{J}$ in
$\overline{R}$ as a $\overline{K}$-vector space is the same. Hence we may work in $\overline{R}$
and show that $\overline{J}$ is of finite codimension in $\overline{R}$.
Let
\[
0\not=W=\rho_0I+\rho_1U+\rho_2V+\rho_3UV\in\overline{J},\quad \rho_i\in \overline{K}[Z].
\]
We shall show that $\overline{J}$ contains a nonzero central element $W_0=\tau I$, $\tau\in \overline{K}[Z]$.
Hence, modulo $\overline{J}$, the elements of $\overline{R}$ have the form
\[
\sigma_0I+\sigma_1U+\sigma_2V+\sigma_3UV,\quad \sigma_i\in \overline{K}[Z],\deg_Z\sigma_i<\deg_Z\tau,
\]
and the factor algebra $\overline{R}/\overline{J}$ is finite dimensional. This is true, if
$\rho_1=\rho_2=\rho_3=0$ in the presentation of the element $W\in\overline{J}$. Hence we may assume that
$W$ is not central. Direct computations give
\[
[W,U]^2=\rho_2(\rho_2+\delta\rho_3)t^2(t^2-\delta\varepsilon)I,
\quad [W,V]^2=\rho_1(\rho_1+\varepsilon\rho_3)t^2(t^2-\delta\varepsilon)I.
\]
Hence, if $\rho_1,\rho_2,\rho_2+\delta\rho_3, \rho_1+\varepsilon\rho_3\not=0$, then $\overline{J}$ contains a nonzero
central element. We have to consider three more cases:

(1) $\rho_1=\rho_2=0$. Hence $\rho_3\not=0$ and $[W,U+V]^2=\rho_3^2t^2(\delta\varepsilon-t^2)I\not=0$.

(2) $\rho_1=0$, $\delta=1$, $\rho_2=-\rho_3\not=0$. Then $[W,U+V]^2=\rho_3^2t^2(\varepsilon-t^2)I\not=0$.
(The case $\rho_2=0$, $\varepsilon=1$, $\rho_1=-\rho_3\not=0$ is similar.)

(3) $\delta=\varepsilon=1$, $\rho_1=\rho_2=\rho_3\not=0$. Then
$[W,U+V]^2=\rho_3^2t^4(1-t^2)I\not=0$.

In all the cases $\overline{J}$ contains a nonzero central element and hence is of finite codimension in $\overline{R}$.
\end{proof}

\begin{remark}
An embedding similar to that in Theorem \ref{Main Theorem} appears also in group theory.
For example, it is well known that for $t\in\mathbb Z$, $t\geq 2$, the matrices
\[
U=\left(\begin{matrix}
1&t\\
0&1\\
\end{matrix}\right),\quad
V=\left(\begin{matrix}
1&0\\
t&1\\
\end{matrix}\right)
\]
generate a free subgroup of $SL_2({\mathbb Z})$.
\end{remark}

\begin{remark}
Our embedding of the monomial algebra
\[
F=K\langle x,y\mid x^2=0,y^2=0\rangle
\]
is the same as the embedding which follows from the proof of Belov \cite{Be}.
When the characteristic of $K$ is different from 2, the embedding of Weiss \cite{W1, W2}
into $M_2(K[v])$ of the algebra
\[
F=K\langle x,y\mid x^2=x,y^2=y\rangle
\]
generated by two idempotents is given by the formulas
\[
x\to
\left(\begin{matrix}
1+v&t\\
1&1-v\\
\end{matrix}\right),\quad
y\to
\left(\begin{matrix}
1+v&-t\\
-1&1-v\\
\end{matrix}\right),
\]
where $t=1-v^2$.
(In fact these matrices are idempotents only up to a multiplicative constant
because they satisfy the equations $x^2=2x$ and $y^2=2y$.
Clearly, this is not essential for the embedding of $F$ into $M_2(K[v])$.)
Then $F$ is isomorphic to the algebra of matrices of the form
\[
\left(\begin{matrix}
f_1+f_3v&t(f_2-f_4v)\\
f_2+f_4v&f_1-f_3v\\
\end{matrix}\right),
\]
where $f_i\in K[t]$. Hence our embedding of $F$ is simpler than that of Weiss.
\end{remark}

\begin{remark}
Fixing an admissible ordering on the set of monomials $\langle x_1,\ldots,x_d\rangle$,
the transfer of combinatorial results for a monomial ideal $J$ of $K\langle x_1,\ldots,x_d\rangle$ to an arbitrary ideal of
$K\langle x_1,\ldots,x_d\rangle$  with the same set of leading monomials as $J$ does not hold automatically.
For example, Irving \cite{I} showed that the algebra with presentation
\[
B=K\langle x,y\mid x^2=0, yxy=x\rangle
\]
satisfies the polynomial identities of $M_n(K)$ (or of $M_n(\overline{K})$ if $K$ is finite)
for a suitable $n$ but cannot be embedded into any matrix algebra
$M_k(C)$ over a commutative algebra $C$. More precisely, it follows from his proof that $B$ satisfies
the polynomial identity $[x_1,x_2][x_3,x_4][x_5,x_6]=0$ which as is well known generates the T-ideal of the $3\times 3$
upper triangular matrices. Hence $B$ satisfies all polynomial identities of $M_3(K)$ (or of $M_3(\overline{K})$ if $K$ is finite).
As a vector space $B$ has a basis consisting of $y^a,y^ax,xy^{a+1},xy^{a+1}x$, $a\geq 0$, which is the same as of the monomial algebra
\[
B_0=K\langle x,y\mid x^2=0, yxy=0\rangle.
\]
By the result of Borisenko \cite{Bo} the algebra $B_0$ can be embedded into a matrix algebra over $K[t]$.
It is interesting to mention that $B_0$ can be realized as a homomorphic image of a subalgebra of $M_3(K[t])$
in the following way. Let $S$ be the algebra generated by
\[
X=\left(\begin{matrix}
0&1&0\\
0&0&1\\
0&0&0\\
\end{matrix}\right),
\quad
Y=\left(\begin{matrix}
0&0&0\\
0&t&0\\
0&0&0\\
\end{matrix}\right).
\]
Then $YXY=0$. Every vector subspace of
$K[t]e_{13}$
is an ideal of $S$. Choosing
\[
J=tK[t]e_{13}=\left(\begin{matrix}
0&0&tK[t]\\
0&0&0\\
0&0&0\\
\end{matrix}\right)
\]
we obtain that $X^2\in J$ and it is easy to see that $B_0\cong S/J$.
\end{remark}

\section*{Acknowledgements}

The first author is grateful to the Department of Mathematics of
the University of Miskolc for the warm hospitality during his
visit when a part of this project was carried out.

\end{document}